\documentclass[12pt,english]{amsart}

\usepackage{amsfonts,amsmath,amssymb,color,txfonts,amsthm,graphics,mathrsfs}
\usepackage{filecontents}
\usepackage{enumitem} 
\usepackage[dvipsnames]{xcolor}

\usepackage[T1]{fontenc}
\usepackage[english]{babel}
\usepackage{color}
\definecolor{atomictangerine}{rgb}{1.0, 0.6, 0.4}
\usepackage[T1]{fontenc}
\usepackage[all]{xy}
\usepackage{setspace}

\usepackage{enumerate}
\usepackage{amscd}
\usepackage{exscale}
\usepackage{latexsym}
\usepackage[all]{xy}
\usepackage{hyperref}
\usepackage{fancyhdr}
\usepackage{layout}
\usepackage{pb-diagram}
\usepackage{bm}

\usepackage[normalem]{ulem}

\usepackage[author={Soli V.}, final]{pdfcomment}

\newtheorem{theorem}{Theorem}

\newtheorem*{mlem}{Four-Point Lemma}
\newtheorem*{m3lem}{Three-Point Lemma}
\newtheorem{lemma}{Lemma}[section]
\newtheorem{corollary}{Corollary}
\newtheorem{conj}{Conjecture}
\newtheorem{cor}[lemma]{Corollary}
\newtheorem*{corollary*}{Corollary}
\newtheorem{proposition}{Proposition}

\usepackage{thmtools}

\theoremstyle{definition}

\newtheorem{remark}[lemma]{Remark}
\newtheorem{example}{Example}

\newtheorem{remark-intro}{Remark}

\newcommand\NN{{\mathbb N}}

\newcommand\QQ{{\mathbb Q}}

\newcommand\ZZ{{\mathbb Z}}

\newcommand\p{{\mathbb P}}

\newcommand\Proj{{\mathbb P}}
\newcommand\PGL{{\mathrm{PGL}}}

\newcommand\tr{\hbox to 1mm  {${}^t \!  $} }

\newcommand{\nc}{\newcommand}
\nc{\apl}{K\cup \{\infty\}}
\nc{\mfp}{\mathfrak{p}}
\nc{\dmfp}{\delta_\mfp}
\nc{\vmfp}{v_\mfp}
\nc{\ace}{\`e }
\nc{\aca}{\`a }
\nc{\aci}{\`i }
\nc{\aco}{\`o }
\nc{\acu}{\`u }
\nc{\pid}{\mathfrak{p} }
\nc{\bdm}{\begin{displaymath}}
\nc{\edm}{\end{displaymath}}
\nc{\beq}{\begin{equation}}
\nc{\eeq}{\end{equation}}
\nc{\dpid}{\delta_{\mathfrak{p}}}
\nc{\vvs}{\textquotedblleft}
\nc{\vvd}{\textquotedblright}
\nc{\os}{\mathcal{O_S}}

\nc{\mcu}{\mathcal{U}}
\nc{\rs}{\sqrt{R_S}}
\nc{\rsu}{\sqrt{R_S^*}}
\nc{\bpm}{\begin{pmatrix}}
\nc{\epm}{\end{pmatrix}}
\nc{\srs}{\sqrt{R_S}}
\nc{\un}{\{1,2,\ldots,n\}}
\DeclareMathOperator{\lcm}{lcm}

\numberwithin{equation}{section}

\title[Scarcity of cycles]{Scarcity of cycles for rational functions over a number field}

\author{Jung Kyu Canci}
\address{Jung Kyu Canci, Universit\"{a}t Basel, Mathematisches Institut, Spiegelgasse $1$, CH-$4051$ Basel}
\email{jungkyu.canci@unibas.ch}
\author{Solomon Vishkautsan}
\address{Solomon Vishkautsan, Scuola Normale Superiore,
Piazza dei Cavalieri 7, 56126 Pisa,
Italy}
\email{wishcow@gmail.com \textrm{or} solomon.vishkautsan@sns.it}

\begin{document}

\begin{abstract}
  We provide an explicit bound on the number of periodic points of a
  rational function defined over a number field, 
  where the bound depends only on the number of
  primes of bad reduction and the degree of the function, and is
  linear in the degree. 
More generally, we show that there exists an explicit uniform bound on the number of
periodic points for any rational function in a given finitely
generated semigroup (under composition) of rational functions of
degree at least 2. We show that under stronger assumptions the
dependence on the degree of the map in the bounds can be removed.
\end{abstract}

\maketitle

\section{Introduction}

In this article we prove the following theorem.

\begin{theorem}\label{MainT}
  Let $K$ be a number field and $S$ a finite set of places of $K$
  containing all the archimedean ones. Let
  $\phi, \psi\colon\Proj_1\to\Proj_1$ be rational maps defined over $K$, where the degree $d_\phi$ of $\phi$ is $\geq 2$. Assume that 
  both maps have good reduction outside $S$.  Then the number of
  $K$-rational periodic points of the composition $\psi\circ\phi$ is bounded by $\kappa d_\phi+\lambda$, for some positive
  integers $\kappa$ and $\lambda$ depending only on the cardinality of
  $S$, and the bound can be effectively computed.
\end{theorem}

We remark that the constants $\kappa$ and $\lambda$ in the theorem depend only on the cardinality of $S$ and thus implicitly on the degree $[K:\QQ]$ but not on the field $K$ itself.

We say that $P\in
\p_1$ is \emph{periodic} for a rational map $\phi\colon\p_1\to\p_1,$
defined over some base field $K$, if there exists a positive integer $n$ such that $\phi^n(P)=P$ (where $\phi^n$ is the $n$-th iteration by composition of the map $\phi$). The minimal such $n$ is called the \emph{period} of $P$.
For a point $P\in\p_1(K)$, the set $O_{\phi}(P)=\{\phi^i(P)\mid{i}\in\NN, i\geq{0}\}$ is called the (forward) \emph{orbit} of $P$ with
respect to $\phi$. We say that a point $P$ is \emph{preperiodic} for $\phi$ when its orbit contains a periodic point,
i.e., when the orbit of $P$ is finite.

 


A rational map $\phi\colon\mathbb{P}_1\to\mathbb{P}_1$ defined over a number field $K$ is said to have \emph{good reduction} at a prime $\pid$ of $K$ if $\phi$ can be written as $\phi=[F(X,Y):G(X,Y)]$ where $F,G\in{R_\pid[X,Y]}$ are homogeneous polynomials of degree $d$, such that the resultant of $F$ and $G$ is a $\pid$-unit, where $R_\pid$ is the localization of the ring of integers of $K$ at $\pid$. For a fixed finite set $S$ of places of $K$ containing all the archimedean ones, we say that $\phi$ has good reduction outside of $S$ if it has good reduction at each place $\pid\notin S$. 

The set of preperiodic points in $\p_1(\bar{K})$ of a rational
map $\phi:\p_1\to\p_1$ of degree $d\geq{2}$ defined over a number field $K$, where $\bar{K}$
is the algebraic closure of $K$, is of bounded height
(this is a special case of Northcott's theorem~\cite{NO}).
Since a number field $K$ possesses the Northcott property (i.e., that
every set of bounded height is finite, cf.~\cite[\S{4.5}]{BG}), the set
of $K$-rational preperiodic points of $\phi$ is finite, so in particular the set of $K$-rational periodic points of $\phi$ is finite too. 
The problem is to find a bound on the number of preperiodic points that depends in a \vvs minimal\vvd\ way on the map $\phi$. One of the main motivations for our research is the following conjecture of Morton and Silverman~\cite{MS1}.

\begin{conj}[Uniform Boundedness Conjecture] Let $\phi:\p_N\to\p_N$ be a morphism of degree $d\geq{2}$ defined over a number field $K$. Then 
the number of $K$-rational preperiodic points of $\phi$ is bounded by a bound depending only on $N$ and the degrees of $K/\mathbb{Q}$ and $\phi$.
\end{conj}

The statement of Theorem~\ref{MainT} is trivially true when $\psi$ is a constant map. If in Theorem~\ref{MainT} we take $\psi$ to be the identity map, then we immediately obtain the following corollary.

\begin{corollary}
  Let $K, S, \phi$ and $d_\phi$ be as in Theorem~\ref{MainT}. Then the
  number of $K$-rational periodic points of $\phi$ is bounded linearly in $d_\phi$, with coefficients that depend only on the cardinality of the set $S$.
\end{corollary}

The techniques used in the proof of Theorem~\ref{MainT} can be extended
to prove uniform boundedness on the number of periodic points in a large class of finitely generated semigroups of rational maps.

\begin{corollary}\label{CorFG}
Let $K$ be a number field and let $I$ be a finitely generated
semigroup, with respect to composition, of rational maps $\p_1\to\p_1$
of degree $\geq 2$ defined over $K$. Then there exists a uniform bound
$B$, depending on $I$, such that any $\phi\in I$ has at most $B$ 
periodic points in $\p_1(K)$.
\end{corollary}

The bound $B$ in Corollary \ref{CorFG} is effectively computable. Let $\mathcal{G}$ be any given finite set of generators of $I$, and let $S$ be a finite set of places of $K$
  containing all the archimedean ones, such that all the elements in $\mathcal{G}$ have good reduction outside of $\mathcal{S}$. Our proof of the corollary provides a bound $B$ depending only on the cardinality of $\mathcal{S}$ and the maximal degree of the elements of $\mathcal{G}$.

We would like to remark that while this work shares some of the techniques used in \cite{C0} and \cite{C}, Corollary \ref{CorFG} does not follow from any of the results therein, and requires the new methods developed in this article.


We also extract the following special case from the proof of Theorem~\ref{MainT}.

\begin{corollary}\label{cor:everywhere-good-reduction}
Let
  $\phi\colon\Proj_1\to\Proj_1$ be a rational map of degree $d\geq 2$
  defined over $\QQ$ with everywhere good reduction 
  (i.e., the map has good reduction at every non-archimedean place of $K$),
then the number of $\QQ$-rational periodic points 
  of $\phi$ is bounded by $d+5$.
\end{corollary}

Let $\phi$ be a monic polynomial of degree at least $2$ with coefficients in $\ZZ$. Note that $\phi$ has
good reduction at every prime. 
This imposes strong conditions on the $\QQ$-rational periodic points of $\phi$. For instance $\phi$ can have
$\QQ$-rational periodic points of period at most $2$ (this fact can be proved with elementary arguments, e.g.\ \cite[Theorem~12.9]{NA2}).
During the preparation of this article, the authors were informed by W.\ Narkiewicz of an unpublished proof of G.\ Baron from 1990,
showing that the number of $\QQ$-rational periodic points of monic polynomials of degree $\geq{2}$ with integer coefficients is at most $d+1$, including the fixed point at infinity, which is better (for the special case of polynomials) than the bound we obtained in Corollary~\ref{cor:everywhere-good-reduction}. For the reader's convenience, we reproduce Baron's result in Section~\ref{sec:baron}. 
Baron's bound is sharp; for example, for any positive integer $d\geq{2}$ the polynomial
\begin{equation}\label{dfixed}
\phi(x)=(x-1)(x-2)\cdots(x-d) + x
\end{equation}
has $d+1$ (fixed) periodic points. 

The following example of Baron shows that the number of periodic points of period $2$ must also depend on the degree of the map. Let $n_1,\ldots,n_d$, with $d\geq 2$, be distinct positive integers. Then the monic polynomial
\begin{equation}\label{dperiod-two}
\phi(x)=\prod_{i=1}^d (x^2-n_i^2)-x
\end{equation}
is of degree $2d$ with everywhere good reduction, and has $2d$ periodic points of period $2$, namely $\pm n_1,\pm n_2,\ldots,\pm n_d$.  

Recall that a point $[a:b]\in\p_1(K)$ is \emph{ramified} (or \emph{critical}) for a rational function $\phi\colon \p_1\to\p_1$ if the order of zero at $[a:b]$ of the algebraic condition $\phi([x:y])=\phi([a:b])$ is greater than $1$. Under this terminology, Baron's special case of monic polynomials can be identified as rational maps with everywhere good reduction having a totally ramified fixed point.
Motivated by this idea of imposing ramified periodic points to obtain smaller bounds, one can in fact eliminate the dependence on the degree in Theorem~\ref{MainT} by making stronger assumptions. 

\begin{theorem}\label{thm:main2}
Let $K, S, \phi$ be as in Theorem~\ref{MainT}. Suppose that there exists a set $\mathcal{A}\subset\p_1(K)$ of periodic points with the property 
\begin{multline}\label{ineq:main2}
\left|\left\{A\in \mathcal{A}\mid \
      \text{\textnormal{$A$ is ramified for $\phi$}}\right\}\right|
+\left| \p_1(K)\cap \left(\phi^{-1}(\phi(\mathcal{A}))\right)\setminus \mathcal{A}\right|\geq 3.
\end{multline}
Then the number of $K$-rational periodic points of $\phi$ is bounded by $3\cdot7^{4s}+3$, where $s=|S|$. 
\end{theorem}

The left hand side of \eqref{ineq:main2} counts the number of ramified points in $\mathcal{A}$ together with the number of non-periodic points whose images are periodic points in $\phi(\mathcal{A})$. 

\begin{example}
Let $f(x)=x^2(x-1)g(x)$ where $g(x)$ is \emph{any} polynomial defined over a number field $K$ with good reduction outside of a finite set of places $S$ containing all the archimedean ones. Then $f(x)$ has at most $3\cdot7^{4s}+3$ periodic points, where $s=|S|$. In fact, take $\mathcal{A}=\{0,\infty\}$; these are two ramified periodic points, and together with $1$ whose image $0$ is in $\phi(\mathcal{A})$, we have at least 3 in the left hand side of \eqref{ineq:main2}.
\end{example}

We obtain the following corollary from the proof of Theorem~\ref{thm:main2}.
\begin{corollary}\label{cor:everywhere-good-reduction2}
Let $\phi\colon\Proj_1\to\Proj_1$ be a rational map of degree $d\geq 2$ defined over $\QQ$ with everywhere good reduction. Suppose there exists a set $\mathcal{A}\subset\p_1(K)$ as in Theorem~\ref{thm:main2}, then the number of $\QQ$-rational periodic points is at most $4$.
\end{corollary}

\begin{remark-intro} \label{rem:ramified-cycles}
  Theorem~\ref{thm:main2} and Corollary~\ref{cor:everywhere-good-reduction2} are still true if instead of ramified periodic points we take periodic points with ramified points in their orbits (see Remark~\ref{rem:ramified-cycles-part-two}). We say that such a point \emph{belongs to a ramified cycle}. In particular, under this modification Corollary~\ref{cor:everywhere-good-reduction2} implies that there is no rational map defined over $\QQ$ with everywhere good reduction and more than four $\QQ$-rational points belonging to a ramified cycle.
\end{remark-intro}

In \cite{MS1} Morton and Silverman proved that if 
   $\phi\colon\Proj_1\to\Proj_1$ is a rational map of degree $d\geq 2$
   defined over a number field $K$, with good reduction at all but $t-2$ primes
   of $K$, and 
   $P\in\mathbb{P}_1(K)$ is a periodic point for $\phi$ with
   period $n$, then
\begin{equation} \label{bound:MS}
n\leq \left(12t\log(5t)\right)^{4\left[K:\mathbb{Q}\right]}.
\end{equation}
The first author studied in \cite{C} the same problem considered by Morton and Silverman, but for preperiodic points. He proved that for $\phi,K$ and $S$ as in Thereom~\ref{MainT}, and $P\in\p_1(K)$ a preperiodic point for $\phi$, we have
\begin{equation} \label{bound:C}
|O_\phi(P)|\leq \left[e^{10^{12}}(s+1)^8(\log(5(s+1)))^8\right]^{s},
\end{equation}
where $s=|S|$.
Note that the bounds in \eqref{bound:MS} and \eqref{bound:C} are independent of the degree of $\phi$. 
 
By using elementary techniques (cf.\ the proof of \cite[Corollary
1.1]{CP}) and the bounds \eqref{bound:MS} and \eqref{bound:C} one can
obtain bounds for the cardinalities of the sets Per$(\phi,K)$ and
PrePer$(\phi, K)$ that are polynomial in the degree $d$ of $\phi$,
with large exponents depending on the cardinality of $S$. In 
Theorem~\ref{MainT} we prove in fact that the cardinality of the set
Per$(\phi,K)$ can be \emph{linearly} bounded with respect to $d$ and coefficients depending only on $|S|$. 

Benedetto in \cite{B} gave a bound for the size of
PrePer$(\phi,K)$ in terms of the set $S$ of places of bad reduction of
$\phi$, for polynomial maps $\phi\in{K[t]}$. Benedetto's bound is of
the form $O(|S|\log|S|)$ where the big--O constant is essentially
$(d^2-2d+2)/\log d$ (of course, the bound also depends on the degree
of $K$ over $\QQ$). Benedetto's techniques are quite different from
ours and involve the filled Julia set associated to the map $\phi$. 
We should also mention that the goal of Benedetto was to minimize
dependence on $|S|$, while we tried to minimize dependence on the
degree $d$. 

To obtain our results, we make use of well known explicit bounds on
the number of 
solutions for equations of
the form 
\begin{equation}\label{nvars-unit-equation}
a_1x_1+\ldots+a_nx_n=1
\end{equation}
with coefficients $a_i\in\mathbb{C}^*$ and the  solutions 
$(x_1,\ldots,x_n)$ are in $\Gamma^n \subset
\left(\mathbb{C}^*\right)^n$, where $\Gamma$ is a finitely generated
subgroup of $\mathbb{C}^*$ (when $\Gamma$ is a group of $S$-units in a
number field,
equation \eqref{nvars-unit-equation} is known as the \emph{$S$-unit equation}
in the literature). When $n>2$ the solutions must also satisfy a
\textit{nondegeneracy} property, such that $\displaystyle{\sum_{i\in{I}}} a_ix_i
\neq 0$ for every nonempty subset $I$ of $\{1,\ldots,n\}$. We will
denote a bound for the number of solutions of
$\eqref{nvars-unit-equation}$ by $c(n,r)$, where $r$ is the rank of
the group $\Gamma$, and in the special case when $n=2$ we denote the bound
by $b(r)$. 

When $\Gamma$ is the group of $S$-units of a number field $K$, the
best known bounds are due to Evertse (see \cite{E84, E95}) : $c(n,s-1)=2^{35n^4s}$ and
$b(s-1)=3\cdot7^{4s}$, where $s=|S|$. However, for technical reasons
explained in section \ref{scc},
we will need to work with extensions of the group of $S$-units, so that
more general bounds are needed. One can use $b(r) = 2^{8(2r+2)}$ (see
\cite{B.S.1}; note that the bound given there is different due to
notational differences and the fact that we need a bound for $ax+by=1$
rather than for $x+y=1$) and $c(n,r)=e^{(6n)^{3n}(nr+1)}$ (see
\cite{E.S.S.1}; same caveat as the previous reference). Since the groups we will be dealing with are of rank
$r=|S|-1$ (see section \ref{scc}) we will use
throughout the article the notation $B(s)=b(s-1)$ and
$C(n,s)=c(n,s-1)$, where $s=|S|$. 
\begin{proposition} \label{prop:estimates}
Let $K, S, \phi,\psi$ and $d_\phi$ be as in Theorem~\ref{MainT}.
Using the notations $B(s)$ and $C(n,s)$ defined above, the numbers $\kappa$ and $\lambda$ of Theorem \ref{MainT} can be chosen as follows: 
$$\kappa=3B(|S|)+13\quad,\quad \lambda=27B(|S|)+C(5,|S|)+6C(3,|S|)+31.$$
\end{proposition}

As in \cite{MS}, we define the \emph{$\mathfrak{p}$-adic chordal valuation} on
$\mathbb{P}_1(K)$ for a finite place $\mathfrak{p}$ in a number field
$K$ by 
\begin{equation*}
\dpid(P_1,P_2) = \nu_\mathfrak{p} (x_1y_2-x_2y_1)-\min\{\nu_\mathfrak{p}(x_1),\nu_\mathfrak{p}(y_1)\}-\min\{\nu_\mathfrak{p}(x_2),\nu_\mathfrak{p}(y_2)\}
\end{equation*}
for points $P_1=[x_1,y_1], P_2=[x_2,y_2]$ in $\mathbb{P}_1(K)$ ($\dpid$ is called a \emph{logarithmic $\mathfrak{p}$-adic distance} function in \cite{MS}, because $\dpid$ is $-log$ of the usual $\mathfrak{p}$-adic chordal metric \cite[\S{2.1}]{SIL1}).

Our Theorem~\ref{MainT} follows from the next lemma, which is
free from dynamical arguments, and might be useful in other
settings as well.

\begin{mlem}\label{lem:4P} Let $K$ be a number field and $S$ a finite set of places of $K$ containing all the archimedean ones. 
Let $\phi\colon \p_1\to\p_1$ be a rational map of degree $d\geq 2$
defined over $K$ with good reduction outside $S$. Let
$A,C,E,G\in\p_1(K)$ be four distinct points such that also the images $\phi(A),\phi(C),\phi(E),\phi(G)$
are distinct. Let $\mathcal{P}$ be the set of points
$P\in \p_1(K)$ satisfying the following four equations for all $\pid\notin{S}$.
\begin{align}\label{main-lemma-eq}
\begin{split}
\dpid(A,P)&=\dpid(\phi(A),\phi(P)),\quad \dpid(C,P)=\dpid(\phi(C),\phi(P)), \\
\dpid(E,P)&=\dpid(\phi(E),\phi(P)),\quad \dpid(G,P)=\dpid(\phi(G),\phi(P)).
\end{split}
\end{align} 
Then $\mathcal{P}$ is finite and 
\beq|\mathcal{P}|\leq \label{bound}(3B(|S|)+13)d+27B(|S|)+C(5,|S|)+6C(3,|S|)+32.\eeq
\end{mlem}

The lemma is proved by showing that the $\mathfrak{p}$-adic chordal
distance properties in the lemma induce unit equations as in
\eqref{nvars-unit-equation}. 

The $S$-unit equation theorem for the equation $u+v=1$, is equivalent
to saying that 
the set of points of the affine curve
$\p_1\setminus \{0,\infty,1\}$ that are $S$-integral with respect to the divisor
$\{0,\infty,1\}$ is finite.
This in turn is equivalent to saying that
the set of points $P\in \p_1(K)$ such that $\dpid(P,Q)=0,$ for all
$Q\in \{0,\infty,1\}$ and all $\pid\notin{S},$ is a finite set. 
Presented in this way, the $S$-unit equation theorem is similar
to our Four-Point Lemma.

Can the Four-Point Lemma be improved? The linear dependence on the degree $d$ in the lemma cannot be eliminated, as can be seen from examples \eqref{dfixed} and \eqref{dperiod-two}. It is also easy to show that no analogous result holds for just one or two points. 
We do not know whether a general three-point lemma with a bound that is polynomial in the degree of the map can be obtained. Following a suggestion of P.\ Corvaja, however,
we were able to prove a useful three-point lemma under stronger hypotheses. The bound obtained is independent of the degree of the map $\phi$, and allows in particular to prove Theorem~\ref{thm:main2}.

\begin{m3lem}\label{3pointsL}
Let $K$ be a number field and $S$ be a finite set of places of $K$ of cardinality $s=|S|$ containing all the archimedean ones. Let $\phi\colon \p_1\to\p_1$ be an endomorphism defined over $K$  of degree $\geq 2$ with good reduction outside $S$.   Let $\mathcal{A}\subset\p_1(K)$ be a set with the property
\begin{multline}\label{inq3pl}\left|\left\{A\in \mathcal{A}\mid \
      \text{\textnormal{$A$ is ramified for $\phi$}}\right\}\right|
+\left| \p_1(K)\cap \left(\phi^{-1}(\phi(\mathcal{A}))\right)\setminus \mathcal{A}\right|\geq 3.\end{multline}
Then the set 
$$\mathcal{P}_\mathcal{A}=\{P\in\p_1(K)\mid \dmfp(P,A)=\dmfp(\phi(P),\phi(A))\ \text{for all $A\in\mathcal{A}$ and $\pid\notin S$}\}$$
is finite and its cardinality is bounded by $3\cdot 7^{4s}$. 
\end{m3lem}
 
\begin{remark-intro}
A set $\mathcal{A}$ as in the Three-Point Lemma exists if and only if (at least) one of the
following conditions is satisfied:
\begin{itemize}
\item $\phi$ has three distinct $K$-rational ramified points.
\item $\phi$ has two distinct ramified points $A_2,A_3\in\p_1(K)$ and
  a third point $A_1\in\p_1(K)$ (not necessarily distinct from $A_2$
  and $A_3$) such that the fiber $\phi^{-1}(\phi(A_1))$ contains at least one point $B\notin\{A_1,A_2,A_3\}$.
\item $\phi$ has one ramified point $A_3\in\p_1(K)$ and there
  exists a set $\mathcal{B}=\{A_i\}_{1\leq i\leq k}\subset \p_1(K)$
  with $1\leq{k}\leq 2$ such that the set $\phi^{-1}(\phi(\mathcal{B}))\setminus\left(\mathcal{B}\cup\{A_3\}\right)$
  contains at least two $K$-rational points. 
\item There exists a set $\mathcal{B}=\{A_i\}_{1\leq i\leq k}\subset\p_1(K)$ with
  $1\leq{k}\leq 3$ such that the set $\phi^{-1}(\phi(\mathcal{B}))\setminus
  \mathcal{B}$ contains at least three $K$-rational points in $\p_1(K)$. 
\end{itemize}
\end{remark-intro}

By applying the \hyperref[3pointsL]{Three-Point Lemma} we also prove the
following theorem.

\begin{theorem}\label{thm:3points}
Let $K, S, \phi$ and $\psi$ be as in Theorem~\ref{MainT}.
 If there exists a set $\mathcal{A}\subset {\rm Per}(\psi\circ\phi,K)$ verifying condition \eqref{inq3pl} with respect to $\phi$ then 
 \beq\label{eq:3points}|{\rm Per}(\psi\circ\phi,K)|\leq 3\cdot 7^{4s} + 3 .\eeq 
\end{theorem}

  \begin{corollary}\label{Cor:d2}
Let $K,S,\phi$ and $\psi$ be as in Theorem~\ref{thm:3points}, such that the degree of $\phi$ is $2$. 
Then 
  $$|{\rm Per}(\psi\circ\phi,K)|\leq 3\cdot 7^{4s} + 3.$$  
  \end{corollary}

In general we have seen (examples \eqref{dfixed} and \eqref{dperiod-two}) that it is not possible to bound the
cardinality of the set ${\rm Per}(\psi,K)$ independently of the degree
of $\psi$. Our corollary affirms that if we take the map
$\psi(x^2)$ instead of $\psi(x)$ (in the affine model), then the cardinality of ${\rm Per}(\psi(x^2),K)$ has a bound independent on the
degree of $\psi$, depending only on the cardinality of the set
$S$, where $\psi$ has good reduction outside $S$. Note that we can obtain a similar bound by applying Theorem~\ref{MainT} with $d=2$, but the Three-Point Lemma provides a better bound.

One of the main ingredients needed in our proofs of Theorems \ref{thm:main2} and \ref{thm:3points} is Corollary \ref{cor:2points}. 
During the preparation of the article, the authors became aware that Sebastian Troncoso was working on similar problems. Troncoso was kind enough to share a preprint with us, 
and we realized that it already contains the same result as in Corollary \ref{cor:2points} (see \cite[Proposition 2.19]{TR1}).




{\color{ForestGreen} 

\normalcolor

We remark that the idea of using unit equation theorems in the setting
of dynamical
systems is originally due to Narkiewicz (\cite{NA}). 

It would be interesting to extend our result and obtain a bound for
the cardinality of PrePer$(\phi,K)$. One needs some improvements of
our methods, but many of our ideas can be useful also in this more
general setting. 

\noindent\textbf{Acknowledgements.} 
We would like to thank Pietro Corvaja for his suggestions in the direction of the Three--Point Lemma. 
We also thank Fabrizio Barroero, Rob Benedetto, 
Laura Capuano, Jan-Hendrik Evertse, Patrick Ingram, Aaron Levin, Olaf Merkert, W{\l}adys{\l}aw Narkiewicz,  Sebastian Troncoso, and Umberto Zannier for
their comments and suggestions. The second author was partially
supported by the ERC-Grant ``Diophantine Problems," No.\ 267273 during the preparation of this article.
Finally, we are grateful to the referee for the comments, corrections and suggestions.
\section{Properties of the $\pid$-adic chordal valuation}
Recall that we defined the $\mathfrak{p}$-adic \emph{chordal valuation} on
$\mathbb{P}_1(K)$ for a finite place $\mathfrak{p}$ in a number field
$K$ by 
\begin{equation*}
\dpid(P_1,P_2) = \nu_\mathfrak{p} (x_1y_2-x_2y_1)-\min\{\nu_\mathfrak{p}(x_1),\nu_\mathfrak{p}(y_1)\}-\min\{\nu_\mathfrak{p}(x_2),\nu_\mathfrak{p}(y_2)\}
\end{equation*}
for points $P_1=[x_1:y_1], P_2=[x_2:y_2]$ in $\mathbb{P}_1(K)$. {Note that $\dmfp(P_1,P_2)=0$ if and only if the two points $P_1$ and $P_2$ are distinct modulo $\pid$.}

\begin{proposition}\label{MS5.2}
Let $K$ be a number field and $\pid$ a finite place of $K$. Let $\phi\colon\Proj_1\to\Proj_1$ be a rational map with good reduction at $\pid$. Then for any $P,Q\in\p_1(K)$ we have
\bdm  \dpid(\phi(P),\phi(Q))\geq \dpid(P,Q).\edm
\end{proposition}
{Proposition \ref{MS5.2} directly follows from \cite[Proposition 5.2]{MS}.}
We provide a short alternative proof of this proposition for the convenience of the reader.

\begin{proof}[Proof of Proposition \ref{MS5.2}]
Up to replacing $\phi$ by $g\circ\phi\circ{f},$ where $f,g$ are appropriate automorphisms in $\textrm{PGL}_2(R_\pid),$ where $R_\pid$ is the localization of the ring of integers of $K$ at $\pid$, we may assume that $P=[0:1]=\phi(P)$. If $Q$ does not equal $P$ modulo $\pid$ we are done because $\dmfp(P,Q)=0$. Therefore we assume now that $Q\equiv P \pmod{\pid}$, so that we can choose coordinates $x_Q, y_Q$ such that $Q=[x_Q:y_Q]$ with $v_\pid(x_Q)=n>0$ and $v_\pid(y_Q)=0$. Since $\phi$ fixes the point $[0:1]$ there exist two polynomial $F,G\in K[X,Y]$ with $\pid$-integral coefficients such that $\phi([X:Y])=[XF(X,Y):G(X,Y)]$ and $F,G$ do not have common factors modulo $\pid$ (also $G(x_Q,y_Q)$ must be a $\pid$-unit). Therefore we have
$$\dmfp(\phi(P),\phi(Q))=v_\pid(x_QF(x_Q,y_Q))\geq v_\pid(x_Q)=\dmfp(P,Q).$$
\end{proof}

\begin{lemma}\label{div}
Let $K$ be a number field and $\pid$ a finite place of $K$. Let $\phi\colon\Proj_1\to\Proj_1$ be a rational map defined over $K$ with good reduction at $\pid$. Let $P,Q\in \Proj_1(K)$ be two periodic points for $\phi$. Then
$$\dpid(P,Q)=\dpid(\phi(P),\phi(Q)).$$
\end{lemma}

\begin{proof} 
Let $n$ and $m$ be the minimal periods for $P$ and $Q$ respectively. 
By Proposition \ref{MS5.2} we have 
$$\dpid(P,Q)\leq\dpid(\phi(P),\phi(Q))\leq \dpid(\phi^i(P),\phi^i(Q)),$$
for any positive integer $i$. Therefore it is enough to prove that
there exists a positive integer $i$ with $\phi^i(P)=P$ and
$\phi^i(Q)=Q$. Take $i=\lcm(n,m)$.
\end{proof}

\begin{lemma}\label{lemmapcomp}
Let $K,\pid,\phi$ be as in Lemma \ref{div} and let
$\psi\colon\p_1\to\p_1$ be a rational map defined over $K$ with good reduction at $\pid$. Let $P,Q\in \Proj_1(K)$ be two periodic points for the composed map $\psi\circ\phi$. Then
\beq\label{PQ=}\dpid(P,Q)=\dpid(\phi(P),\phi(Q)).\eeq
\end{lemma}
\begin{proof}
The composition $\psi\circ\phi$ has good reduction at $\pid$ (cf.\
Silverman~\cite[Theorem 2.18]{SIL1}). From Lemma \ref{div} and Proposition \ref{MS5.2} we have
$$\dpid(P,Q)\leq \dpid(\phi(P),\phi(Q))\leq \dpid(\psi(\phi(P)),\psi(\phi(Q)))=\dpid(P,Q).$$
Thus (\ref{PQ=}) holds.
\end{proof}

\begin{remark}
In general if $P,Q$ are periodic points for $\psi\circ
\phi$, it does not imply that they are also periodic for
$\phi$. Therefore \eqref{PQ=} does not follow directly from Lemma
\ref{div}. Also, $\psi\circ
\phi$ having good reduction at $\pid$ does not imply that $\phi$ and $\psi$ have good
reduction at $\pid$ (see \cite[Remark~2.19(c)]{SIL1}).
\end{remark} 

Let $A$ be a point in $\p_1(K)$. A point $P\in\p_1(K)$ is said to be \emph{$S$-integral} with respect to $A$ if and only if $\dmfp(P,A)=0$ for all $\pid\notin{S}$. Let $\mathcal{A}$ be a finite set of points in $\p_1(K)$. We say that a point $P\in\p_1(K)$ is $S$-integral with respect to $\mathcal{A}$ if $P$ is integral with respect to $A$ for all $A\in\mathcal{A}$. We also say that $P$ is an \emph{$S$-integer} in $\p_1\setminus\mathcal{A}$.

\begin{lemma}\label{lem:rami_p} Let $K$ be a number field and $S$ a finite set of places of $K$ containing all the archimedean ones. Let $\phi\colon\Proj_1\to\Proj_1$ be a rational map defined over $K$ with good reduction outside $S$.
Let $Q\in \p_1(K)$ be a ramification point of $\phi$ and let
$P\in\p_1(K)$ be a point distinct from $Q$ such that 
$$\dmfp(P,Q)=\dmfp(\phi(P),\phi(Q))$$
for all $\pid\notin S$.  Then $\dpid(P,Q)=0$ for all $\pid\notin S$, i.e., $P$ is an $S$--integer with respect to $Q$.
\end{lemma}

\begin{proof} 
{Let $\pid\notin S$. For all $f\in \PGL_2(R_\pid)$
$$\dmfp(P_1,P_2)=\dmfp(f(P_1),f(P_2))$$
holds for all $P_1,P_2\in\p_1(K)$} (i.e., the $\pid$-adic
distance is preserved under automorphisms of $\p_1$ with good reduction at $\pid$). 
For any map $f\in \PGL_2(R_\pid)$ let $B=f^{-1}(Q)$ and let
$A=f^{-1}(P)$ and choose an automorphism $g\in\PGL_2(R_\pid)$ such that $g(\phi(f(B)))=B$, then
\begin{align*} 
\dmfp(A,B) &=\dmfp(f(A),f(B))\\
&=\dmfp(P,Q)\\
&=\dmfp(\phi(P),\phi(Q))\\
&=\dmfp(\phi(f(A)),\phi(f(B)))\\
&=\dmfp(g(\phi(f(A))),B).
\end{align*}
Hence by taking $B=0$ and $g\circ\phi\circ f$ instead of $\phi$ we can
assume that $Q=[0:1]$ and $Q$ is fixed by $\phi$ (note that
$g\circ\phi\circ f$ has good reduction at $\pid$ and $B$ is ramified
for $g\circ\phi\circ f$). Since $Q=[0:1]$ is ramified for $\phi$ and
$\phi$ has good reduction at $\pid$, there exist two homogeneous
polynomials $r,s\in R_\pid[X,Y]$ such that 
\begin{equation*}
\phi([X:Y])=[X^2r(X,Y):s(X,Y)], 
\end{equation*}
where $s(0,1)$ is a $\pid$-unit. 

Suppose $n$ is a positive integer such that $\dmfp(P,Q)=n$. Therefore
there exist $x_P,y_p\in R_\pid$ such that $P=[x_P:y_P]$ with
$v_\pid(x_P)=n$ and $v_\pid(y_P)=0$. By the good reduction of $\phi$
at $\pid$ we have that the two polynomials $X^2r(X,Y)$ and $s(X,Y)$
have no common factors modulo $\pid$ so that $v_\pid(s(x_P,y_P))$ must
be $0$, and therefore $s(x_P,y_P)$ is a $\pid$-unit. We get
$$n=\dmfp(P,Q)=\dmfp(\phi(P),\phi(Q)) = \dmfp(\phi(P),Q) \geq 2n.$$
Contradiction. Therefore $\dpid(P,Q)=n=0$ as required.
\end{proof}


%

\begin{cor}\label{cor:rami_cycle}
Let $K,S,\phi$ be as in Lemma \ref{lem:rami_p}. 
Let $P,Q\in \p_1(K)$ be distinct periodic points of $\phi$ such that
the orbit of $Q$ contains a ramified point. Then $\dpid(P,Q)=0$ for all $\pid\notin S$, i.e., $P$ is an $S$-integer with respect to $Q$. 
\end{cor}

\begin{proof}
Let $\pid\notin S$.  Let $k\geq{0}$ be such that $\phi^k(Q)$ is a ramified (periodic) point.
Then $\phi^k(P)$ and $\phi^k(Q)$ are distinct and satisfy the conditions of Lemma \ref{lem:rami_p} with $\phi^k$ instead of $\phi$.
Therefore $\dpid(\phi^k(P),\phi^k(Q))=0$. By Lemma~\ref{div} we get 
$$\dpid(P,Q) = \dpid(\phi^k(P),\phi^k(Q))=0.$$
\end{proof}


%
%

\begin{lemma}\label{lem:tail_dist} Let $K,S,\phi$ be as in Lemma \ref{lem:rami_p}. 
 Let $P,Q\in
  \p_1(K)$ be distinct points such $\phi(P)=\phi(Q)$. Let $R\in\p_1(K)$ be a point such that $R\neq P$ and
\beq\label{Scond} \dmfp(P,R)=\dmfp(\phi(P),\phi(R))\eeq
for all $\pid\notin S$. Then $\dpid(Q,R)=0$ for all $\pid\notin S$, i.e., $R$ is an $S$--integer with respect to $Q$. 
\end{lemma}

\begin{proof}Let $\pid\notin S$. 
Suppose that $\dmfp(Q,R)>0$. Since $\phi$ has good reduction at $\pid$
we have that $\dmfp(\phi(Q),\phi(R))>0$ (by Proposition~\ref{MS5.2}). Hence 
$$ \dmfp(P,R) = \dmfp(\phi(P),\phi(R)) = \dmfp(\phi(Q),\phi(R)) > 0.$$
By the triangle inequality proved in \cite[Proposition 5.1]{MS} we have
$$\dmfp(P,Q)\geq \min\{\dmfp(P,R),\dmfp(Q,R)\}>0.$$
Therefore the reduction of $P$ modulo $\pid$ is a
ramified point of the reduced map $\phi_\pid$. With a proof similar to
that of Lemma \ref{lem:rami_p} we obtain 
$\dmfp(P,R)=0$, contradicting our assumption $\dmfp(Q,R)>0$. 
\end{proof}

\begin{remark}
As for Lemma~\ref{lem:rami_p}, the requirements of Lemma~\ref{lem:tail_dist} can only be satisfied by a
map of degree $d\geq{2}$.
\end{remark}

%

%

\begin{cor}\label{cor:2points} Let $K,S,\phi$ be as in Lemma \ref{lem:rami_p}. 
Let $P\in\p_1(K)$ be a
periodic point, and let $Q\in\p_1(K)$ such that $P\neq{Q}$ and
$\phi(P)=\phi(Q)$ (i.e., $Q$ is a tail point of length $1$ for
$\phi(P)$). If $R$ is any periodic point distinct from $P$ then $R$ is an $S$-integer with respect to $Q$. 
\end{cor}

\begin{proof}
Apply Lemma \ref{div} and Lemma \ref{lem:tail_dist}.
\end{proof}

\section{$S$-coprime integral coordinates}\label{scc}
Let $P\in \p_1(K)$ where $K$ is an arbitrary number field. Let $S$ be
a finite set of places containing all the archimedean ones.
There exist $a,b\in R_S$ such that $P=[a:b]$. We say that the
homogeneous coordinates $[a:b]$ are
\emph{$S$-coprime integral coordinates} for $P$ if $\min\{v_\pid(a), v_\pid(b)\}=0$
for each $\pid\notin S$.  

The main property of $S$-coprime integral coordinates that we will need is the following.
\begin{lemma} \label{lemma:s-coprime}
Let $P_1,P_2\in\p_1(K)$ be points that can be written in $S$-coprime
integral coordinates $P_1 = [x_1:y_1]$ and $P_2 = [x_2 : y_2]$. Then
for any $\mathfrak{p}\notin{S}$
$$ \dpid(P_1,P_2) = \nu_\mathfrak{p} (x_1y_2-x_2y_1).$$
\end{lemma}
\begin{proof}
This is clear from the definitions of $S$-coprime integral coordinates and the $\mathfrak{p}$-adic chordal valuation.
\end{proof}

Unfortunately, if the ring $R_S$ is not a principal ideal domain, there exist points in $\p_1(K)$ that do not have $S$-coprime integral coordinates. 
We could avoid this problem  by taking  a larger set $\mathbb{S}$ of places of $K$ containing $S$, such 
that the ring $R_{\mathbb{S}}$ is a principal ideal domain.  Indeed,
the class number of $R_S$ (i.e., the order of the fractional ideal
class group) is finite (e.g., see Corollary in Chapter 5 in
\cite{M.1}). We denote by $h$ the class number of $R_S$. By a simple
inductive argument, we can choose $\mathbb{S}$ such that
$|\mathbb{S}|\leq s+h-1$. By working with $\mathbb{S}$, we will obtain
bounds depending also on $h$. Instead, we use the same argument as in~\cite{C} (also in
\cite{CP}) in order to avoid the presence of $h$ in
our bounds. 

Let $H$ be a subgroup of an abelian group $G$, then the set 
$$\sqrt{H} =  \{g\in{G} | g^m\in{H} \text{ for some positive integer } m\}$$ 
is a subgroup called the \emph{radical
  closure} of $H$ in $G$ (cf.\ \cite[17.3.1]{KM}).
Let $L$ be a finite extension of $K$, let $S_{L}$ be the set
of places of $L$ lying over the places in $S$ and let $R_{S_{L}}$ and 
$R_{S_{L}}^*$ be the ring of $S_{L}$-integers and the group of
$S_{L}$-units in $L$, respectively. Now $\sqrt{R_S^\ast}$, the
radical closure of 
$R_S^\ast$ in $L^*$, has the property that  $\sqrt{R_{S}^\ast}=R_{S_{L}}^*\cap
 {\sqrt{K}}^*$, where $\sqrt{K^*}$ is the radical closure of $K^*$ in $L^*$. Moreover, $\sqrt{R_{S}^\ast}$ is a subgroup of $L^\ast$ of rank $|S|-1$. 

It is possible to choose an extension $L/K$ such that $R_{S_{L}}$ is a principal ideal domain.
This is again a consequence of the finiteness of the class group of $K$, and the construction of the minimal extension $L$ with the required property is described in detail in \cite{C} (after the statement of Theorem B in that article). 
Any point $P\in\mathbb{P}_1(L)$ can thus be written in $S_{L}$-coprime integral coordinates (by $S_{L}$-coprime integral coordinates we mean $S$-coprime integral coordinates with $S=S_{L}$). 
Furthermore, the coordinates can be normalized in such a way that given any two points $P_1=[x_1:y_1]$ and $P_2=[x_2:y_2]$ (written in these normalized coordinates), we have 
$$x_1y_2-x_2y_1\in\sqrt{K^*}.$$ This normalization is described in the proof of Theorem 1 in \cite{C}. We call such a choice of coordinates \emph{$S$-radical coprime integral coordinates} (this type of coordinates was named \emph{$\sqrt{R_S}$-coprime integral coordinates} in \cite{CP}). 

\begin{lemma} \label{lemma:4-points}
Let $P_1,P_2,P_3,P_4\in{\p_1(K)}$, and let $L$ and $S_{L}$ be as discussed above. Consider $P_1,P_2,P_3,P_4$ as points in $\p_1(L)$. Then
\begin{equation} \label{lemma-hat-S:eq1}
\delta_{\hat{\mathfrak{p}}}(P_1,P_2) = \delta_{\hat{\mathfrak{p}}}(P_3,P_4)
\end{equation}
for each $\hat{\mathfrak{p}}\notin{S_{L}}$ if and only if 
there exist $S$-radical coprime integral coordinates 
$$P_1 = [x_1:y_1], P_2=[x_2:y_2], P_3=[x_3:y_3], P_4=[x_4:y_4]$$
such that
\begin{equation} \label{lemma-hat-S:eq2}
x_1y_2-x_2y_1 = u(x_3y_4-x_4y_3),
\end{equation}
for some $u\in\sqrt{R_{S}^\ast}$.
\end{lemma}

\begin{proof}
Since $P_1,P_2,P_3,P_4\in{\p_1(K)}$ we can choose $S$-radical coprime integral coordinates such that $ x_1y_2-x_2y_1, x_3y_4-x_4y_3 \in \sqrt{K^*} $. At the same time, it is clear from \eqref{lemma-hat-S:eq1} and Lemma~\ref{lemma:s-coprime} that there exists a $u\in{R_{S_{L}}^*}$ satisfying \eqref{lemma-hat-S:eq2}. Therefore $u\in{R_{S_{L}}^*\cap{\sqrt{K^*}}} = \sqrt{R_S^*}$. The other implication is trivial.
\end{proof}

\begin{remark} \label{mainrem}
In the proof of Theorem~\ref{MainT}, the finitely generated group $\sqrt{R_S^*}$ takes the place of the group $\Gamma$ in the unit equation \eqref{nvars-unit-equation}.
Since this is not necessarily the group of $S$-units of a number field, we will not be able to use Evertse's bounds for $S$-unit equations (described in the introduction). 
\end{remark}

\section{Proof of the Four-Point Lemma} 

\begin{proof}[Proof of the Four-Point Lemma]
Let $K,S$ and $\phi$ be as in the statement of the lemma (see also Remark \ref{mainrem}), and let $L$ be the extension of $K$ and $S_{L}$ the set of places as described in Section \ref{scc}. 
Let 
$$A=[a_1:b_1], C=[c_1:d_1], E=[e_1:f_1], G=[g_1:h_1]$$
and 
$$\phi(A)=[a_2:b_2], \phi(C)=[c_2:d_2], \phi(E)=[e_2:f_2], \phi(G)=[g_2:h_2]$$
be written in $S$-radical coprime integral coordinates. 

We can assume that none of the images
$\phi(A),\phi(C),\phi(E),\phi(G)$ is the point $[0:1]$, as
otherwise we can conjugate $\phi$ by a suitable automorphism
$\alpha\in{PGL_2(R_S)}$ and move these points away from $[0:1]$.
We replace $\phi$ by the conjugate map $\alpha\circ\phi\circ
\alpha^{-1}$ and the points $A,C,E$ and
$G$ by the points
$\alpha(A),\alpha(C),\alpha(E)$ and $\alpha(G)$, respectively. 
This does not cause problems since $\alpha\circ\phi\circ\alpha^{-1}$
still has good reduction outside $S$, the degree of the map does not
change and the $\pid$-adic chordal valuation remains invariant under the change in coordinates{, via an automorphism with good reduction at $\pid$}.

Let $P\in\mathcal{P}$ be a general point 
and let $[x_1:y_1]$ and $[x_2:y_2]$ be $S$-radical coprime integral coordinates for $P$
and $\phi(P)$, respectively (notice that we do not rule out the
possibility that $P$ is equal to $\phi(P)$ or to any of the points
$A,C,E$ and $G$). We assume in what follows that $P\neq [1:0]$ (i.e.,
$[1:0]\notin\mathcal{P}$), and then add $1$ to the bound that we obtain with $P\neq [1:0]$. 


By Lemma \ref{lemma:4-points}, the four equalities in \eqref{main-lemma-eq} are satisfied if \normalcolor
and only if there exist four units $u_i\in \sqrt{R_S^*}, i\in\{1,2,3,4\}$ such that the
following system of equations is satisfied
\beq\label{sys} 
\begin{cases}
a_1y_1-b_1x_1=u_1(a_2y_2-b_2x_2)\\
c_1y_1-d_1x_1=u_2(c_2y_2-d_2x_2)\\
e_1y_1-f_1x_1=u_3(e_2y_2-f_2x_2)\\
g_1y_1-h_1x_1=u_4(g_2y_2-h_2x_2)
\end{cases}
\eeq
Note that in the system in \eqref{sys} we may assume that one of the
$u_i$ is 1, for example $u_4=1$. 
In fact, it is enough to take $u_4x_1, u_4y_1$ instead of $x_1, y_1$, respectively, and replace $u_i$ by $u_i/u_4$ for $i\in\{1,2,3\}$ (note that $[u_4x_1 : u_4y_1]$ are still $S$-radical coprime integral coordinates for $P$). 

By assumption, the system in \eqref{sys} has a nonzero solution $(x_1,y_1,x_2,y_2)\neq (0,0,0,0)$. Therefore the rank of the matrix
$$M=\begin{pmatrix}
 {a_1} & {b_1} &-{a_2} {u_1} & {b_2} {u_1} \\
 {c_1} & {d_1} &-{c_2} {u_2} & {d_2} {u_2} \\
 {e_1} & {f_1} &-{e_2} {u_3} & {f_2} {u_3} \\
 {g_1} & {h_1} &-{g_2} & {h_2} \\
\end{pmatrix}
$$
associated to the system in \eqref{sys} has to be less than $4$. By
considering the $2\times 2$ minors in the first two columns one sees
that each pair of rows of the matrix $M$ is linearly independent, 
since the points $A,C,E$ and $G$ are distinct. This implies
that the rank of $M$ is at least $2$.
  		

\begin{description}[leftmargin=0cm, labelindent=0cm,topsep=0pt]
\item[Case 1] In this case we consider all cases where the fourth row
  in \eqref{sys} is a linear combination of two other rows in
  \eqref{sys}. Suppose first that the fourth row in \eqref{sys} is a
  linear combination of the first two rows. We consider the
  $(3,4)$ minor of $M$:
\beq
\label{rank2}
\left|\begin{matrix}
 {a_1} & {b_1} &-{a_2} {u_1}\\
 {c_1} & {d_1} &-{c_2} {u_2} \\
  {g_1} & {h_1} &-{g_2} \\
\end{matrix}\right|
= a_2(d_1g_1-c_1h_1)u_1+c_2(a_1h_1-b_1g_1)u_2+g_2(b_1c_1-a_1d_1) = 0.
\eeq
The coefficients $a_2(d_1g_1-c_1h_1), c_2(a_1h_1-b_1g_1)$ and
$g_2(b_1c_1-a_1d_1)$ are all non-zero, since $A,C,E$ and $G$ are distinct points and
their images are different from $[0:1]$. Thus equation \eqref{rank2} is a unit equation of the form \eqref{nvars-unit-equation} in the two units
$u_1,u_2\in{\sqrt{R_S^*}}$. Therefore we have at most $B(|S|)$
solutions $(u_1,u_2)\in\left(\sqrt{R_S^*}\right)^2$ satisfying the equation
\eqref{rank2}. 
Let $(u_1,u_2)\in\left(\sqrt{R_S^*}\right)^2$ be one of these solutions. By
considering the first and last equations of the system \eqref{sys}, we
obtain the system
\begin{equation} \label{first-and-last-eq}
\begin{cases}
\displaystyle u_1a_2\frac{y_2}{y_1}-u_1b_2\frac{x_2}{y_1}=a_1-b_1\frac{x_1}{y_1}\\
\displaystyle g_2\frac{y_2}{y_1}-h_2\frac{x_2}{y_1}=g_1-h_1\frac{x_1}{y_1}
\end{cases}
\end{equation}
(recall that $y_1\neq 0$ since
  $P\neq{[1:0]}$). If we solve \eqref{first-and-last-eq} for
  $\frac{y_2}{y_1}$ and $\frac{x_2}{y_1}$ in terms of
  $\frac{x_1}{y_1}$ we obtain $i,j,k,l\in L$ such that 
\[
[x_2:y_2] = \left[i+j \displaystyle\frac{x_1}{y_1}:k+l\displaystyle\frac{x_1}{y_1}\right].
\]
The equality 
\beq\label{pxy1}\phi([x_1:y_1])=[x_2:y_2]=\left[i+j
  \frac{x_1}{y_1}:k+l\frac{x_1}{y_1}\right]\eeq
can be satisfied for at most $d+1$ points $[x_1:y_1]$. 
Therefore there are at most $(d+1)B(|S|)$ pairs of points
$[x_1:y_1],[x_2:y_2]$ satisfying $\phi([x_1:y_1])=[x_2:y_2]$ and
\eqref{sys}, and such that the fourth row in \eqref{sys} is a
linear combination of the first two rows. Note that here we have used
the condition $d\geq 2$. Indeed, since $d\geq 2$, the identity in  \eqref{pxy1} cannot be
satisfied for infinitely many points $[x_1:y_1]$. 
\pdfcomment{Summary: In this case we proved some things in
  general. x2/y2 is determined by u1 and x1/y1. In turns this means
  there are at most (d+1) solutions for x1/y1 for each value of
  u1. This means that solving for u1 solves the finiteness problem.}

By considering the three
possible subcases in \vvs Case 1\vvd, we see that there are at most 
$3(d+1)B(|S|)$
pairs of points $[x_1:y_1],[x_2:y_2]$
satisfying $\phi([x_1:y_1])=[x_2:y_2]$ and \eqref{sys}, and such that
the fourth row in \eqref{sys} is a linear combination of any other two rows in \eqref{sys}. 
\pdfcomment{Should we show the two other cases? Maybe comment that
  they are basically the same}

\item[Case 2] In this case we assume that the fourth row in
  \eqref{sys} is not a linear combination of any other two rows in \eqref{sys}. 
This means that the rank of the matrix $M$ must be exactly $3$, so
that there is at most one solution $(x_1/y_1,x_2/y_1,y_2/y_1)$ 
determined by each set of values of $(u_1,u_2,u_3)$ (recall that $y_1\neq 0$ since
  $P\neq{[1:0]}$). Moreover, if such a solution exists it is in fact determined by
  any pair $(u_i,u_j)$ for distinct $i,j$ in $\{1,2,3\}$, since
  we assumed that the fourth equation of $\eqref{sys}$ is independent of
  any two other equations.

The determinant of the matrix $M$ is zero, as the space of solutions
of the system in \eqref{sys} has positive dimension by the assumption
of the existence of a point $P$ in $\mathcal{P}$. Thus we have
\begin{multline}\label{eq6a}
(b_2g_2-a_2h_2)(d_1e_1-c_1f_1)u_1+(a_1f_1-b_1e_1)(d_2g_2-c_2h_2)u_2+\\
+(b_1c_1-a_1d_1)(f_2g_2-e_2h_2)u_3+(b_2c_2-a_2d_2)(f_1g_1-e_1h_1)u_1u_2+\\
+(a_2f_2-b_2e_2)(d_1g_1-c_1h_1)u_1u_3+(d_2e_2-c_2f_2)(b_1g_1-a_1h_1)u_2u_3=0.
\end{multline}
By the assumptions on the points $A,C,E,G$, each of the coefficients
of the units $u_1,u_2,u_3,u_1u_2,u_1u_3,u_2u_3$ in \eqref{eq6a} is
nonzero.
 
\item[Subcase where in \eqref{eq6a} there are no vanishing subsums] 
The triple \newline $(u_1,u_2,u_3)$ assumes at
  most $C(5,|S|)$ values in $\left(\sqrt{R_S^*}\right)^3$, since we can divide \eqref{eq6a} by $u_1$ to get a unit equation in five variables. Therefore the
  system \eqref{sys} provides at most $C(5,|S|)$ solutions for
  $(x_1/y_1,x_2/y_1,y_2/y_1)$. Thus there are at most $C(5,|S|)$ pairs of 
  points $[x_1:y_1], [x_2:y_2]$ defined over $K$ with $\phi([x_1:y_1])=[x_2:y_2]$ satisfying
  this subcase of \vvs Case 2\vvd. 

Now we take into consideration the subcases where in \eqref{eq6a}
there are vanishing subsums. These can be sorted into three types:
the subcase ``3+3'' where the sum in \eqref{eq6a} can be divided
into two vanishing subsums each one with three addends; the subcase
``2+2+2'' where the sum in \eqref{eq6a} can be divided into three
vanishing subsums; finally, the subcase ``4+2'' containing all
other possible vanishing subsums.

To ease notation we denote 
\begin{align*}
\alpha=(b_2g_2-a_2h_2),&&\beta=(d_1e_1-c_1f_1),&&\gamma=(a_1f_1-b_1e_1),&&\delta=(d_2g_2-c_2h_2),\\
\eta=(b_1c_1-a_1d_1),&&\nu=(f_2g_2-e_2h_2),&&\mu=(b_2c_2-a_2d_2),&&\lambda=(f_1g_1-e_1h_1),\\
\epsilon=(a_2f_2-b_2e_2),&&\rho=(d_1g_1-c_1h_1),&&\omega=(d_2e_2-c_2f_2),&&\zeta=(b_1g_1-a_1h_1),
\end{align*}
so that \eqref{eq6a} assumes the form
\begin{equation} \label{eq6a-greek}
\alpha\beta u_1+\gamma\delta u_2+ \eta\nu u_3+\mu\lambda u_1u_2 +\epsilon\rho u_1u_3+\omega\zeta u_2u_3=0.
\end{equation}

\item[Subcase \vvs 3+3\vvd] There are exactly 10 possible pairs of
  vanishing subsums in this subcase. Let us first assume that the following two subsums vanish:
\beq\label{3+3a}\begin{cases}
\alpha\beta u_1+\gamma\delta u_2+\eta\nu u_3=0\\
\mu\lambda u_1u_2+\epsilon\rho u_1u_3+ \omega\zeta u_2u_3=0.
\end{cases}\eeq
The above system is equivalent to:
\beq\label{sys3+3}\begin{cases}\displaystyle
-\frac{\alpha\beta}{\eta\nu} \frac{u_1}{u_3}-\frac{\gamma\delta}{\eta\nu}\frac{u_2}{u_3}=1\\\displaystyle
-\frac{\epsilon\rho}{\mu\lambda}\frac{u_3}{u_2}-  \frac{\omega\zeta}{\mu\lambda}\frac{u_3}{u_1}=1.
\end{cases}\eeq
There are at most $2$ pairs
$(u_1/u_3,u_2/u_3)$ in $\left(\sqrt{R_S^*}\right)^2$ satisfying
\eqref{sys3+3}. 
Indeed, if we set $x=u_1/u_3$ and $y=u_2/u_3$, and denote the four coefficients in the system by $a,b,c$ and $d$, then \eqref{sys3+3} assumes the following form.
\beq\label{sys3+3a}\begin{cases}\displaystyle
ax+by=1 \\
dx+cy=xy.
\end{cases}\eeq
This system has at most two possible solutions (intersection of a line with a non-degenerate conic). 
Fix one of these two solutions $(m_1,m_2)\in
\left(\sqrt{R_S^*}\right)^2$. Therefore we have $u_1=m_1u_3$ and
$u_2=m_2u_3$. By considering the first, second and fourth equation in the system \eqref{sys} we obtain  
\beq\label{sysu3}
\begin{cases}
a_1y_1-b_1x_1-m_1u_3a_2y_2+m_1u_3b_2x_2=0\\
c_1y_1-d_1x_1-m_2u_3c_2y_2+ m_2u_3d_2x_2=0\\
g_1y_1-h_1x_1-g_2y_2+h_2x_2=0
\end{cases}
\eeq
where the variables are $x_1,y_1,x_2,y_2$ and $u_3$ is considered a
free parameter in $\sqrt{R_S^*}$. 
Each nontrivial solution of the system
\eqref{sysu3} must satisfy equalities of the following form:
\begin{equation}\label{solu3}
\frac{x_1}{y_1} = q_1(u_3),\quad \frac{x_2}{y_1} = q_2(u_3),\quad \frac{y_2}{y_1} = q_3(u_3),
\end{equation}
where $q_1,q_2,q_3$ are quadratic rational functions with coefficients
in $L$. This can be seen by applying Cramer's rule to $\eqref{sysu3}$,
solving for \newline
$(x_1/y_1,x_2/y_1,y_2/y_1)$. Since the denominators of
$q_1,q_2$ and $q_3$ are all equal to the determinant of
$\eqref{sysu3}$ with respect to $(x_1/y_1,x_2/y_1,y_2/y_1)$, we get
that $x_1/y_1$ and $x_2/y_2$ are also at most quadratic rational
functions in $u_3$ with coefficients in $L$ (if $y_2$ is zero so we have $[x_2:y_2]=[1:0]$; recall also that we set $[x_1:y_1]\neq [1:0]$). Solving the equation $\phi([x_1:y_1])=[x_2:y_2]$ with respect to $u_3$ leads to a polynomial equation of degree $\leq 2d+2$ in $u_3$. Thus, there are at most $2d+2$ possible values for $u_3$ (as in \vvs Case 1\vvd, we have used the fact that $\deg(\phi)=d\geq 2$) for each solution $(m_1,m_2)$ of \eqref{sys3+3}. Each value of $u_3$ determines a unique pair $(u_1,u_2)$ from $(m_1,m_2)$ so there are at most $4d+4$ triples $(u_1,u_2,u_3)$ that satisfy system \eqref{3+3a}.
Thus there at most $4d+4$ pairs $[x_1:y_1],[x_2:y_2]\in\p_1(K)$ of points with $\phi([x_1:y_1])=[x_2:y_2]$, such that $x_1,x_2,y_1,y_2$ satisfy \eqref{sys} and the associated units $u_1,u_2,u_3$ satisfy the vanishing subsums in \eqref{3+3a}.

Let us now assume that the following two subsums vanish:
\beq\label{3+3b}\begin{cases}
\alpha\beta u_1+\gamma\delta u_2+\mu\lambda u_1u_2=0\\
\eta\nu u_3+\epsilon\rho u_1u_3+ \omega\zeta u_2u_3=0.
\end{cases}\eeq
After cancelling out $u_3$ form the second equation, we see that there are 
at most $2$ solutions
$(u_1,u_2)\in\left(\sqrt{R_S^*}\right)^2$ for
\eqref{3+3b} (once again, we have an intersection of a line with a non-degenerate conic). Each pair $(u_1,u_2)$ determines uniquely a triple
$(x_1/y_1, y_2/y_1, x_2/y_1),$ since the first,
  second and fourth equations of system \eqref{sys} are of rank
  $3$. This triple uniquely determines the ordered pair $([x_1:y_1],
  [x_2:y_2])$ so that there are at most $2$ pairs of points $[x_1:y_1],[x_2:y_2]$ with  $\phi([x_1:y_1])=[x_2:y_2]$ satisfying \eqref{3+3b}. 

We obtain the same bound for the vanishing subsums
$$\begin{cases}
\alpha\beta u_1+\eta\nu u_3+\epsilon\rho u_1u_3=0\\
\gamma\delta u_2+\mu\lambda u_1u_2+ \omega\zeta u_2u_3=0,
\end{cases}
$$

$$\begin{cases}
\gamma\delta u_2+\eta\nu u_3+ \omega\zeta u_2u_3=0\\
\alpha\beta u_1+\mu\lambda u_1u_2+\epsilon\rho u_1u_3=0.\end{cases}
$$
We assume now that the following two subsums vanish:
\beq\label{3+3c}\begin{cases}
\alpha\beta u_1+\gamma\delta u_2+ \omega\zeta u_2u_3=0\\
\eta\nu u_3+\mu\lambda u_1u_2+\epsilon\rho u_1u_3=0.
\end{cases}\eeq

There are at most $B(|S|)$ solutions
$(u_1/u_2,u_3)\in \left(\sqrt{R_S^*}\right)^2$ for the first equation in
\eqref{3+3c}. For each such solution $(m_1,m_2)$, the second equation of \eqref{3+3c} becomes 
$$\eta\nu m_2+\mu\lambda m_1u_2^2+\epsilon\rho m_1m_2 u_2=0.$$
Thus there are at most two possible values of $u_2$ for each solution
$(m_1,m_2)$. Therefore there are at most $2B(|S|)$ triples
$(u_1,u_2,u_3)\in \left(\sqrt{R_S^*}\right)^3$ satisfying \eqref{3+3c}. These
triples determine, at most $2B(|S|)$ pairs of points $[x_1:y_1],[x_2:y_2]\in \p_1(K)$ with $\phi([x_1:y_1])=[x_2:y_2]$ satisfying \eqref{sys}.
Note that there are 6 subcases as the one given by the system
\eqref{3+3c}; namely where one equation contains monomials in $u_i$
and $u_j$, and the other equation contains a monomial in $u_iu_k$ with
$\{i,j,k\}=\{1,2,3\}$.

Therefore for the subcase \vvs 3+3\vvd \ there are in total at most 
$$(4d+4)+3\cdot{2}+12B(|S|) = 4d + 12B(|S|) + 10$$
pairs of points $[x_1:y_1],[x_2:y_2]\in\p_1(K)$ with $\phi([x_1:y_1])=[x_2:y_2]$.

\item[Subcase \vvs 4+2\vvd] 
There are exactly 15 subcases of the form \vvs 4+2\vvd. 

We first consider the case where the vanishing subsums are:
\beq\label{4+2a}\begin{cases}
\alpha\beta u_1+\gamma\delta u_2+\eta\nu u_3+\mu\lambda u_1u_2=0\\
\epsilon\rho u_1u_3+ \omega\zeta u_2u_3=0.
\end{cases}\eeq
We eliminate $u_1$ to get 
\begin{equation} \label{el4+2a}
\left(-\frac{\alpha\beta\omega\zeta}{\epsilon\rho} +\gamma\delta\right) u_2+\eta\nu u_3-\frac{\mu\lambda\omega\zeta}{\epsilon\rho}u_2^2=0.
\end{equation}
Note that
$-\frac{\alpha\beta\omega\zeta}{\epsilon\rho}+\gamma\delta\neq 0$
as this would imply that the subsum $\eta\nu u_3+\mu\lambda u_1u_2$ vanishes, but we are in the case \vvs 4+2\vvd. \pdfcomment{this is not clear
  to me.}
Therefore there are at most $B(|S|)$ solutions $(u_2,
u_3/u_2)$ for equation \eqref{el4+2a}. Together with the second equation in
\eqref{4+2a}, each solution determines a triple $(u_1,u_2,u_3)$ which
in turn determines uniquely the points $[x_1:y_1]$ and
$[x_2:y_2]$. Thus the vanishing subsums in \eqref{4+2a} provide at
most $B(|S|)$ pairs of points $[x_1:y_1],[x_2:y_2]\in\p_1(K)$
with  $\phi([x_1:y_1])=[x_2:y_2]$ satisfying \eqref{3+3b}. There
exactly 3 pairs of vanishing subsums as in  \eqref{4+2a}; namely where
the ``short'' equation contains a monomial in $u_iu_j$ and a monomial
in $u_ju_k$ with $\{i,j,k\}=\{1,2,3\}$.

Let us assume now that the following two subsums vanish
\beq\label{4+2b}\begin{cases}
\alpha\beta u_1+\gamma\delta u_2+\epsilon\rho u_1u_3+ \omega\zeta u_2u_3=0\\
\eta\nu u_3+\mu\lambda u_1u_2=0.
\end{cases}\eeq
We use the second equation to eliminate $u_3$ and get
$$\alpha\beta u_1+\gamma\delta u_2-\frac{ \epsilon\rho\mu\lambda}{\eta\nu} u_1^2u_2-\frac{\omega\zeta\mu\lambda}{\eta\nu}   u_1u_2^2=0$$
This equation admits at most $C(3,|S|)$ solutions
$(\frac{u_2}{u_1}, u_1u_2,u_2^2)\in\left(\sqrt{R_S^*}\right)^3$. Each such
triple determines two possible values for $u_2$, so there are at most
$2C(3,|S|)$ solutions $(u_1,u_2)\in\left(\sqrt{R_S^*}\right)^2$. Thus for the
subcase of vanishing subsums as in \eqref{4+2b} we have at most
$2C(3,|S|)$ pairs of points $[x_1:y_1],[x_2:y_2]\in\p_1(K)$
with  $\phi([x_1:y_1])=[x_2:y_2]$. Note that there are exactly three
vanishing subsums of the same form as \eqref{4+2b}, namely where the ``short'' vanishing sum contains a monomials in $u_i$ and $u_ju_k$ with $\{i,j,k\}=\{1,2,3\}$. 

We assume now that the two vanishing subsums are the following ones:
\beq\label{4+2c}\begin{cases}
\alpha\beta u_1+\gamma\delta u_2+\mu\lambda u_1u_2+ \omega\zeta u_2u_3=0\\
\eta\nu u_3+\epsilon\rho u_1u_3=0.
\end{cases}\eeq
This system is equivalent to
\begin{equation} \label{4+2c:simplified}
\begin{cases}\displaystyle
 -\frac{\alpha\beta\eta\nu}{\epsilon\rho}+\left(\gamma\delta -\frac{\mu\lambda\eta\nu}{\epsilon\rho}\right)u_2+ \omega\zeta u_2u_3=0\\\displaystyle
 u_1=-\frac{\eta\nu}{\epsilon\rho}.
\end{cases}
\end{equation}
There are at most $B(|S|)$ solutions
$(u_2,u_2u_3)\in\left(\sqrt{R_S^*}\right)^2$ for the first equation in
\eqref{4+2c:simplified}. Therefore in the case of vanishing subsums as
in \eqref{4+2c} we have at most $B(|S|)$ pairs of points
$[x_1:y_1],[x_2:y_2]\in\p_1(K)$ with
$\phi([x_1:y_1])=[x_2:y_2]$. Note that there are exactly 6 vanishing
subsums of the form \eqref{4+2c}, 
namely where the ``short'' vanishing sum contains monomials in $u_i$ and $u_iu_k$ with $i,k\in\{1,2,3\}, i\neq k$. 

Finally, we assume now that the two vanishing subsums are the following ones:
\beq\label{4+2d}\begin{cases}
\alpha\beta u_1+\gamma\delta u_2=0\\
\eta\nu u_3+\mu\lambda u_1u_2+\epsilon\rho u_1u_3+ \omega\zeta u_2u_3=0.
\end{cases}\eeq
We use the first equation to eliminate $u_1$ and get
$$ \eta\nu
u_3-\frac{\mu\lambda\gamma\delta}{\alpha\beta}u_2^2+\left(-\frac{\epsilon\rho\gamma\delta}{\alpha\beta}+
\omega\zeta\right) u_2u_3=0,$$
There are at most $B(|S|)$ solutions
$\left(\frac{u_2^2}{u_3},u_2\right)\in\left(\sqrt{R_S^*}\right)^2$ for this
equation. Therefore in the case of vanishing subsums of the form \eqref{4+2d} we have at most $B(|S|)$ pairs of points $[x_1:y_1],[x_2:y_2]\in\p_1(K)$ with  $\phi([x_1:y_1])=[x_2:y_2]$. Note that there are exactly 3 vanishing subsums of the form\eqref{4+2c}, namely where the ``shortest'' vanishing sum contains monomials in $u_i$ and $u_j$ with $i,j\in\{1,2,3\}, i\neq j$. 

Therefore, in the whole case ``4+2'' provides a total of at most 
$$3B(|S|)+6C(3,|S|)+6B(|S|)+3B(|S|)=6C(3,|S|) + 12B(|S|)$$
pairs of projective points $[x_1:y_1],[x_2:y_2]\in\p_1(K)$
with $\phi([x_1:y_1])=[x_2:y_2]$ satisfying \eqref{sys}.

\item[Subcase \vvs 2+2+2\vvd]
We assume that the three vanishing subsums are the following ones:
\beq\label{2+2a}\begin{cases}
\alpha\beta u_1+\gamma\delta u_2=0\\
\eta\nu u_3+\mu\lambda u_1u_2=0\\
\epsilon\rho u_1u_3+ \omega\zeta u_2u_3=0
\end{cases}\eeq
Note that we should have $\det\bpm \alpha\beta &\gamma\delta \\
\epsilon\rho & \omega\zeta\epm =0$ since $u_1,u_2,u_3$ cannot be zero,
so that we can drop the third equation from the system. 
The system \eqref{2+2a} is therefore equivalent to 
$$\begin{cases}\displaystyle
u_2=-\frac{\alpha\beta}{\gamma\delta} u_1\\\displaystyle
\eta\nu u_3=-\frac{\mu\lambda\alpha\beta}{\gamma\delta} u_1^2\\
\end{cases}$$
By using the same argument as the one for the first subcase ``3+3'', i.e., system \eqref{3+3a}, we see
that at most $2d+2$ values are possible for $u_1$. Therefore we
conclude that for the case of vanishing subsums in \eqref{2+2a}
there are at most $2d+2$ pairs of points
$[x_1:y_1],[x_2:y_2]\in\p_1(K)$ with
$\phi([x_1:y_1])=[x_2:y_2]$. Note that there are exactly 3 vanishing
subsums of the form \eqref{2+2a}, namely where one vanishing sum
contains monomials in $u_i$ and $u_k$ and another vanishing subsum
contains monomials in $u_iu_j$ and $u_ku_j$ with $\{i,j,k\}=\{1,2,3\}$.

We assume now that the three vanishing subsums are the following ones:
\beq\label{2+2c}\begin{cases}
\alpha\beta u_1+\gamma\delta u_2=0\\
\eta\nu u_3+\epsilon\rho u_1u_3=0\\
\mu\lambda u_1u_2+ \omega\zeta u_2u_3=0.
\end{cases}\eeq
This system admits at most one solution
$(u_1,u_2,u_3)\in\left(\sqrt{R_S^*}\right)^3$. Thus there is at most one pair
$[x_1:y_1],[x_2:y_2]\in\p_1(K)$ of points satisfying
$\phi([x_1:y_1])=[x_2:y_2]$ and systems \eqref{sys} and
\eqref{2+2c}. Note that there are exactly 6 vanishing subsums of the
shape as in \eqref{2+2c}, namely where one subsum contains monomials
in $u_i$ and $u_j$ and another subsum contains monomials in $u_k$ and $u_iu_k$ with $\{i,j,k\}=\{1,2,3\}$. 

We assume now that the three vanishing subsums are the following ones:
\beq\label{2+2b}\begin{cases}
\alpha\beta u_1+\epsilon\rho u_1u_3=0\\
\gamma\delta u_2+ \omega\zeta u_2u_3=0\\
\eta\nu u_3+\mu\lambda u_1u_2=0
\end{cases}\eeq
Assuming there is a solution, we drop the second equation. The system
\eqref{2+2b} is then equivalent to
$$\begin{cases}\displaystyle
 u_3=-\frac{\alpha\beta}{\epsilon\rho}\\\displaystyle
\eta\nu u_3+\mu\lambda u_1u_2=0
\end{cases}
$$
Since $u_3$ is uniquely determined by this system, we can use the
third and fourth rows of system \eqref{sys} to express
$[x_2:y_2]$ in terms of $[x_1:y_1]$ and $u_3$. Then, as in
Case 1, the condition $\phi([x_1:y_1])=([x_2:y_2])$ implies that there at most $d+1$ solutions for
$[x_1:y_1]$. Thus there are at most $d+1$ pairs of
points $[x_1:y_1],[x_2:y_2]\in\p_1(K)$ with
$\phi([x_1:y_1])=([x_2:y_2])$ satisfying \eqref{sys} and
\eqref{2+2b}. Note that there are 3 vanishing subsums of the form
\eqref{2+2b}, namely where one subsum contains a monomial in $u_i$ and
$u_ju_k$, another contains a monomial in $u_j$ and the third contains a
monomial in $u_ju_i$, with $\{i,j,k\}=\{1,2,3\}$. 

We assume now that the three vanishing subsums are the following ones:
\beq\label{2+2d}\begin{cases}
\gamma\delta u_2+\epsilon\rho u_1u_3=0\\
\alpha\beta u_1+ \omega\zeta u_2u_3=0\\
\eta\nu u_3+\mu\lambda u_1u_2=0.
\end{cases}\eeq
This system is equivalent to
$$\begin{cases}\displaystyle
u_2=-\frac{\epsilon\rho}{\gamma\delta}u_1u_3\\\displaystyle
u_3^2=\frac{\alpha\beta\gamma\delta}{\epsilon\rho\omega\zeta}\\
\displaystyle
u_1^2=\frac{\eta\nu\gamma\delta}{\lambda\mu\epsilon\rho}
\end{cases}
$$
Therefore there are at most $4$ solutions $(u_1,u_2,u_3)\in \left(\sqrt{R_S^*}\right)^3$ for the system \eqref{2+2d}. Thus there are at most $4$ pairs of points $[x_1:y_1],[x_2:y_2]\in\p_1(K)$ with $\phi([x_1:y_1])=([x_2:y_2])$ satisfying \eqref{sys} and \eqref{2+2d}. 

It remains to consider only the following two pairs of vanishing subsums
\beq\label{2+2e}\begin{cases}
\alpha\beta u_1+\epsilon\rho u_1u_3=0\\
\gamma\delta u_2+\mu\lambda u_1u_2 =0\\
\eta\nu u_3+\omega\zeta u_2u_3=0\end{cases}\ \ ,\ \ \begin{cases}
\alpha\beta u_1+\mu\lambda u_1u_2=0\\
\gamma\delta u_2+ \omega\zeta u_2u_3=0\\
\eta\nu u_3+\epsilon\rho u_1u_3=0\end{cases},\eeq
both of which have at most one solution $(u_1,u_2,u_3)\in \left(\sqrt{R_S^*}\right)^3$.

Therefore there are at most $2$ pairs of points $[x_1:y_1],[x_2:y_2]\in\p_1(K)$ with $\phi([x_1:y_1])=([x_2:y_2])$ satisfying \eqref{sys} and one of the systems in \eqref{2+2e}. 

Therefore for the subcase ``2+2+2'' there is a total of at most 
$$3(2d+2)+6+3(d+1)+4+2=9d+21$$
pairs of projective points $[x_1:y_1],[x_2:y_2]\in\p_1(K)$ with $\phi([x_1:y_1])=[x_2:y_2]$.
\end{description}

Putting together all possible subcases of \vvs Case 1\vvd\ , \vvs
Case 2\vvd\  and $P=[1:0]$ we find that there are at most 
\begin{multline*}
3(d+1)B(|S|) + C(5,|S|) + (4d + 12B(|S|) + 10) +\\+ 6C(3,|S|) + 12B(|S|) +
9d+21 +1
\end{multline*}
points in $\mathcal{P}$, which is the number in (\ref{bound}).
\end{proof}

\section{Proof of the Three-Point Lemma}

In order to prove the Three-Point Lemma, we will use the following well known result.
\begin{lemma}\label{Sint3p}
Let $K$ be a number field and $S$ a finite set of places of $K$ of
cardinality $|S|=s$ containing all the archimedean ones. Let $A_1,A_2,
A_3\in\p_1(K)$ be three distinct points. Then the number of $S$-integral points of $\p_1\setminus\{A_1,A_2, A_3\}$ is finite and is bounded by $3\cdot 7^{4s}$.
\end{lemma}
Lemma \ref{Sint3p} is a direct consequence of Evertse's explicit bound
on the number of solutions for the $S$-unit equation in two variables (see \cite[Theorem
1]{E84}), by writing the system of linear equations corresponding to the $S$-integral conditions $\dmfp(P,A_i)=0$ for all $i\in\{1,2,3\}$ and $\pid\notin S$.

\begin{proof}[Proof of the Three-Point Lemma] 
The inequality \eqref{inq3pl}, together with Lemmas~\ref{lem:rami_p}
and~\ref{lem:tail_dist}, implies that there exist three distinct points $Q_1,Q_2,Q_3\in \phi^{-1}(\phi(\mathcal{A}))$ such that the set $\mathcal{P}_\mathcal{A}$ is contained in the set of $S$-integral points in $\p_1\setminus\{Q_1,Q_2,Q_3\}$.  Now it is enough to apply Lemma \ref{Sint3p}.
\end{proof}

\section{Proofs of the theorems and their corollaries}

 \begin{proof}[Proof of Theorem~\ref{MainT} and Proposition~\ref{prop:estimates}]

{We assume that there are at least four distinct periodic
points in Per$(\psi\circ\phi,K)$, as otherwise we are done. We choose
four distinct periodic points and denote them by $A,C,E$ and $G$,
respectively.
Since $A,C,E,G$ are four distinct periodic points, the points \newline
$\phi(A),\phi(C), \phi(E),
\phi(G)$ are also four distinct points. }

Let $P\in \p_1(K)$ be a periodic point for $\psi\circ\phi$. By Lemma
\ref{lemmapcomp} the four equalities in \eqref{main-lemma-eq} hold. Therefore 
$${\rm Per}(\psi\circ\phi,K)\subseteq \mathcal{P},$$
where $\mathcal{P}$ is the set defined in the \hyperref[lem:4P]{Four-Point Lemma}. 
Therefore, by the \hyperref[lem:4P]{Four-Point Lemma} we may choose $\kappa=3B(|S|)+13$ and $\lambda=27B(|S|)+C(5,|S|)+6C(3,|S|)+32$.
However, we may then decrease $\lambda$ by $1$,
since we can first move all the periodic points away from $[0:1]$ by
conjugating $\phi$ by an appropriate element in $PGL_2(R_S)$, so
that even if $[0:1]$ is in $\mathcal{P}$ it is not periodic (recall
that in the proof of the \hyperref[lem:4P]{Four-Point Lemma} we always
add $1$ to include the possibility that $[0:1]$ is in $\mathcal{P}$).
\end{proof}

\begin{proof}[Proof of Corollary \ref{CorFG}]
Let $\mathcal{G}=\{\phi_1,\ldots,\phi_n\}$ be a minimal set (with
respect to inclusion) of generators of the semigroup $I$, such that
$\phi_1,\ldots,\phi_n$ have degrees $d_1,\ldots,d_n$ respectively. Let
$S$ be the set of places of $K$ containing all the archimedean ones
and all the places of bad reduction of all the maps in $\mathcal{G}$. 
Therefore each map in $I$ has good reduction outside $S$. 
Furthermore each map in $I$ can be presented as $\psi\circ \phi_i$ for
some $i\in \{1,\ldots,n\}$ and $\psi$ is either the identity map or $\psi\in I$. In both cases we have that $\psi$ has good reduction outside $S$. Let $d=\max_{1\leq i\leq n}\{d_i\}$. We may apply our Theorem~\ref{MainT} to obtain:
$$\left|{\rm Per}(\psi\circ\phi_i,K)\right|\leq \kappa d+\lambda$$
where the numbers $\kappa$ and $\lambda$ are the ones given in the proof of Theorem~\ref{MainT}.
\end{proof}

\begin{proof}[Proof of Corollary \ref{cor:everywhere-good-reduction}]
Here $S$ contains only the archimedean absolute value, and the group
of $S$-units is just $\{\pm{1}\}$. Let $a,b\in\mathbb{C}^*$, then the
unit equation $ax+by=1$ has only one solution with $x,y\in\{\pm{1}\}$,
since only one of the equations $a+b=1, a-b=1, b-a=1, a+b=-1$ can be
satisfied simultaneously. Therefore we can take $B(|S|)$ to be $1$.

In Case 1 in the proof of the \hyperref[lem:4P]{Four-Point Lemma}, 
we only need the case where the fourth row of \eqref{sys}
is a linear combination of the first and second rows. This subcase of
Case 1 has at most $d+1$ solutions. Suppose now that the first, second
and fourth rows are linearly independent. As mentioned in Case 2 of
the proof, any solution of the system is then determined uniquely by the
values of $u_1$ and $u_2$, but these have at most $4$ values. Thus
Case 2 collapses to checking if \eqref{sys} has solutions with these
four values of $(u_1,u_2)$.

Thus in total there are at most $d+5$ periodic points defined over
$\mathbb{Q}$.
\end{proof}

{\begin{proof}[Proof of Theorem~\ref{thm:main2}]
Lemma~\ref{lemmapcomp} implies that the set of $K$-rational periodic points of $\phi$ is contained in the set $\mathcal{P}_{\mathcal{A}}$ in the Three-Point Lemma. Thus the theorem follows immediately from the Three-Point Lemma.
\end{proof}}

\begin{proof}[Proof of Corollary~\ref{cor:everywhere-good-reduction2}]
As remarked after the statement of the Three--Point Lemma, we may assume that $|\mathcal{A}|\leq 3$. 
Over $\mathbb{Q}$, with $S$ containing only the archimedean place, we have $R_S^*=\{\pm 1\}$. Thus, for given $a,b\in\mathbb{Q}$, the $S$-unit equation $ax+by=1$ has at most one solution for $x,y\in R_S^*$. In the proof of Theorem~\ref{thm:main2}, the bound for the cardinality of the set of $K$--rational periodic points for $\phi$ is $B+3$, where $B$ is the bound for the solutions in $R_S^*$ of he $S$-unit equation $ax+by=1$. In the setting of the present corollary, the number $B$ is 1, that proves the corollary.
\end{proof}

\begin{remark} \label{rem:ramified-cycles-part-two}
As pointed out in Remark~\ref{rem:ramified-cycles} in the introduction, the ramified periodic points in Theorem~\ref{thm:main2} and Corollary~\ref{cor:everywhere-good-reduction2} can be replaced by periodic points with ramified points in their orbit. This can be proved by combining Corollaries~\ref{cor:rami_cycle} and~\ref{cor:2points} together with Lemma~\ref{Sint3p}.
\end{remark}

\begin{proof}[Proof of Theorem~\ref{thm:3points}]
By Lemma~\ref{lemmapcomp} we have 
$${\rm Per}(\psi\circ\phi,K)\subseteq \mathcal{P}_\mathcal{A}$$
where $\mathcal{P_\mathcal{A}}$ is the set defined in \hyperref[3pointsL]{Three-Point Lemma}.
Therefore, by the \hyperref[3pointsL]{Three-Point Lemma} the inequality \eqref{eq:3points} holds.
\end{proof}

  \begin{proof}[Proof of Corollary~\ref{Cor:d2}]
If $\psi\circ\phi$ has less than three $K$-rational periodic points we are done. Suppose that $\psi\circ\phi$ has at least three $K$-rational periodic points. We choose as $\mathcal{A}$ a set with three $K$-rational periodic points for $\psi\circ\phi$. By Lemma
\ref{lemmapcomp} we have 
$${\rm Per}(\psi\circ\phi,K)\subseteq \mathcal{P}_\mathcal{A}.$$
Since $\phi$ has degree 2, each point in $A\in \mathcal{A}$ either is a ramification point for $\phi$ or the fiber over $\phi(A)$ admits a $K$--rational preperiodic points that is not periodic, so not in $\mathcal{A}$. Hence we conclude by applying \hyperref[3pointsL]{Three-Point Lemma}.
  \end{proof}

\section{Polynomials with integer coefficients} \label{sec:baron}

In this section, we reproduce the main results from G.\ Baron's unpublished article, providing alternative shorter proofs (we emphasize that the lemmas below already appeared in Baron's original manuscript).

\begin{lemma}\label{lemB1}
Let $f\in\ZZ[x]$ be a monic polynomial of degree $\geq 2$. Suppose that $(a,b)$ is a $2$-cycle in $\QQ$ and $e\in \QQ$ is fixed point. Then $2e=a+b$. 
\end{lemma} 
\begin{proof}
Since the degree is $\geq 2$, then $a,b,e\in\ZZ$ (all periodic points of $f$ are roots of monic polynomials with integer coefficients, and therefore must be integral over $\ZZ$). The polynomial map $f$ has good reduction at any prime of $\ZZ$, so that it has good reduction outside $S=\{\infty\}$, the set containing only the unique archimedean place of $\QQ$. The only $S$-units are $\pm 1$. By applying Lemma~\ref{div} we have 
$$a-e=u\left(f(a)-f(e)\right)=u(b-e),$$ 
where $u$ is an $S$-unit. Since $a\neq{b}$ we must have $u=-1$, so that $2e=a+b$.
\end{proof}

\begin{lemma}\label{lemB2}
Let $f\in\ZZ[x]$ be a monic polynomial of degree $\geq 2$. Suppose that $(a,b), (c,d)$ are two distinct $2$-cycles in $\QQ$. Then $a+b=c+d$. 
\end{lemma}
\begin{proof}
As in the proof of the previous lemma, the map $f$ has good reduction outside $S=\{\infty\}$ and $a,b,c,d\in\ZZ$. By applying Lemma~\ref{div} we get 
\beq\label{B2}
a-c= \pm(b-d)\ ,\ a-d=\pm (b-c)
\eeq
Suppose $a-c=b-d$. Then from the second identity in \eqref{B2} we have $d\in\{b,c\}$, contradiction. Therefore we must have $a-c=-(b-d),$ i.e., $a+b=c+d$ as required. 
\end{proof}

\begin{theorem}
Let $f\in\ZZ[x]$ be a monic polynomial of degree $d\geq 2$. Then
$$|{\rm Per}(f,\QQ)|\leq d+1.$$
\end{theorem}
\begin{proof}
Suppose that all the periodic points of $f$ in $\p_1(\QQ)$ are fixed. Then the bound $d+1$ holds, as fixed points are roots of $f(x)-x$. Suppose that $f$ has a periodic point $a\in \QQ$ that is not fixed. Then the period of $a$ is two, and we set $h=a+f(a)$. By Lemmas~\ref{lemB1} and~\ref{lemB2} we have $b+f(b)=h$ for any periodic point $b\in\QQ$ (so in $\ZZ$). Therefore any such $b$ is a root of the polynomial equation $f(x)-x-h=0$, which has degree $d$. Thus there are at most $d$ periodic points in $\QQ$. Together with the point at infinity we get $d+1$. 
\end{proof}

\providecommand{\bysame}{\leavevmode\hbox to3em{\hrulefill}\thinspace}
\providecommand{\MR}{\relax\ifhmode\unskip\space\fi MR }
\providecommand{\MRhref}[2]{%
  \href{http://www.ams.org/mathscinet-getitem?mr=#1}{#2}
}
\providecommand{\href}[2]{#2}

\end{document}